\documentclass[12pt]{article}
\usepackage{amsxtra,amssymb,amsthm,amsmath,latexsym}

\theoremstyle{plain}

\newtheorem{theorem}{Theorem}[section]
\newtheorem*{remark3.2}{Remark 3.2}
\newtheorem*{remark3.3}{Remark 3.3}
\numberwithin{equation}{section}

\def\nd{\noindent}

\def\oH{\buildrel\circ\over H}
\def\oH1{\buildrel\circ\over H\kern-.02in{}^1}

\def\d{\delta}

\begin{document}
%A.G.Ramm, On a new notion of the solution to an ill-posed problem 
%\begin{titlepage}

\title{                  
On a new notion of the solution to an ill-posed problem
  \thanks{key words: ill-posed problems, regularizer,
stable solution of ill-posed problems
    }
   \thanks{AMS2010 subject classification: 47A52, 65F22, 65J20 }
}

\author{
A.G. Ramm\\
Mathematics Department, Kansas State University, \\
 Manhattan, KS 66506-2602, USA\\
ramm@math.ksu.edu\\
%http://www.math.ksu.edu/\,$\widetilde{\ }$\,ramm\\
}

\date{}

\maketitle\thispagestyle{empty}

\begin{abstract}
%\footnote{Math subject classification: 65J15, 47H17, 58C15}
%\footnote{key words: regularizer, ill-posed problem }

A new understanding of the notion of the stable solution to 
ill-posed problems is proposed. 
The new notion is more realistic than the old one
and better fits the practical computational needs. 
A method for
constructing stable solutions in the new sense is proposed and justified.
The basic point is: in the traditional definition of the stable solution 
to an ill-posed problem $Au=f$, where $A$ is a linear or nonlinear 
operator in a Hilbert space $H$,   it is 
assumed that the noisy data 
$\{f_\delta, \delta\}$ are given, $||f-f_\delta||\leq \delta$, and a 
stable solution $u_\d:=R_\d f_\d$ is defined by the
relation $\lim_{\d\to 0}||R_\d f_\d-y||=0$, where $y$ solves the 
equation
$Au=f$, i.e., $Ay=f$. In this definition $y$ and $f$ are unknown. 
Any $f\in B(f_\d,\d)$ can be the exact data, where $B(f_\d,\d):=\{f: 
||f-f_\delta||\leq \delta\}$.

The new notion of the stable  solution excludes the unknown
$y$ and $f$ from the definition of the solution.

\end{abstract}

%\end{titlepage}

\section{Introduction}
Let
\begin{equation} \label{1.1} Au=f, \end{equation}
where $A:H\to H$ is a linear closed operator, densely defined in a Hilbert 
space $H$.
Problem \eqref{1.1} is called ill-posed if $A$ is not a homeomorphism
of $H$ onto $H$, that is, either equation \eqref{1.1} does not have a 
solution,
or the solution is non-unique, or the solution does not depend on $f$
continuously. Let us assume that \eqref{1.1} has a solution, possibly 
non-unique. Let $N(A)$ be the null space of $A$, and $y$ be the unique 
normal solution to \eqref{1.1}, i.e., $y\perp N(A)$.
Given noisy data $f_\delta$, $\|f_\delta-f\|\leq \delta$,
one wants to construct a stable approximation $u_{\delta}:=R_\d f_\d$ of 
the solution
$y$, $\|u_\delta-y\|\to 0$ as $\delta \to 0$.

Traditionally (see, e.g., \cite{M}) one calls a family of
operators $R_h$ a regularizer for problem \eqref{1.1} (with not 
necessarily linear operator $A$) if

\nd a) $R_hA(u)\to u$ as $ h\to 0$ for any $u\in D(A)$,\\
\nd b) $R_hf_\delta$ is defined for any $f_\delta\in H$
and there exists $h=h(\delta)\to 0$ as $\delta\to 0$ such that
\begin{equation} \label{ast} \|R_{h(\delta)} f_\delta-y\| \to 0 
\tag{{$\ast$}} \hbox { as } \delta\to 0. \end{equation}

In this definition $y$ is fixed and \eqref{ast} must hold for any
$f_\delta\in B(f,\delta):=\left\{ f_\delta:\|f_\delta-f\|\leq \delta 
\right\}$.

In practice one does not know the solution $y$ and the exact data $f$.
The only available information is a family $f_\delta$ and some {\it a
priori information about $f$ or about the solution $y$}.
This a priori information often
consists of the knowledge that $y\in {\mathcal K}$, 
where ${\mathcal K}$ is a compactum in $H$.
Thus
$$y\in S_\delta
  :=\left\{ v:\|A(v)-f_\delta\|\leq\delta,\ v\in {\mathcal K}\right\}.$$
We assume that the operator $A$ is known exactly, and we always
assume that $f_\delta\in B(f,\delta)$, where $f=A(y)$.

{\bf Definition:} {\it We call a family of operators $R(\delta)$ a 
regularizer if}
\begin{equation} \label{1.2}
 \sup_{v\in S_\delta} \|R(\delta) f_\delta-v\|
 \leq \eta(\delta)\to 0 \hbox{\quad as\quad } \delta \to 0. \end{equation}
There is a {\it crucial difference} between our new Definition 
\eqref{1.2}
and the standard definition \eqref{ast}:

{\it In \eqref{ast} $u$ is fixed,
while in \eqref{1.2} $v$ is an arbitrary element of $S_\delta$ and the 
supremum of the norm in \eqref{1.2} over all such $v$ 
must tend to zero as $\delta\to 0$.}

The new definition is more realistic and better fits computational 
needs
because not only the solution $y$ to \eqref{1.1} satisfies the inequality
$\|Ay-f_\delta\|\leq \delta$,
but any $v\in S_\d$ satisfies this inequality
$\|Av-f_\delta\|\leq \delta$, $v\in {\mathcal K}$.
The data $f_\delta$ may correspond
to any $f=Av$, where  $v\in S_\delta$, and not only to $f=Ay$, where $y$ 
is a solution of equation (1.1).
Therefore it is more natural  to use definition \eqref{1.2} than \eqref{ast}.

Our goal is to illustrate the practical difference between these two
definitions, and to construct regularizer in the sense
\eqref{1.2} for problem \eqref{1.1} with an arbitrary, not necessarily 
bounded, linear operator $A$, which is closed and densely defined in 
$H$. This is done in Section 2. In Section 1 this 
is done for a class of equations \eqref{1.1} with
nonlinear operators $A:X\to Y$, where $X$ and $Y$ are Banach spaces. 
In this case we assume that

\nd A1) $A:X\to Y$ {\it is a closed, nonlinear, injective map},
$f\in {\mathcal R}(A)$, $ {\mathcal R}(A)$ {it is the range of} $A$,

and

\nd A2) $\phi: D(\phi)\to [0,\infty)$, $\phi(u)>0$ {\it if} $u\not= 0$,
 $D(\phi)\subseteq D(A)$,
{\it the sets} ${\mathcal K}={\mathcal K}_c:=\{v:\phi(v)\leq c\}$
{\it are compact in} $X$ {\it for every} $c=const>0$, {\it and if}
$v_n\to v$, {\it then}
$\phi(v)\leq \liminf_{n\to\infty} \phi(v_n)$.

The last inequality holds if $\phi$ is {\it lower semicontinuous}.
In Hilbert spaces and  in reflexive Banach spaces norms are
lower semicontinuous.

Let us give some examples of equations for which assumptions A1) and A2) 
are satisfied.

\nd{\bf Example 1.}
$A$ is a linear injective compact operator, $f\in{\mathcal R}(A)$,
$\phi(v)$ is a norm on $X_1\subset X$, where $X_1$ is densely
imbedded in $X$, the embedding $i:X_1\to X$ is compact, and
$\phi(v)$ is lower semicontinuous.

\nd{\bf Example 2.}
$A$ is a nonlinear injective continuous operator $f\in{\mathcal R}(A)$,
$A^{-1}$ is not continuous, $\phi$ is as in Example 1.

\nd{\bf Example 3.}
$A$ is linear, injective, densely defined, closed operator,
$f\in{\mathcal R}(A)$, $A^{-1}$ is unbounded, $\phi$ is as in Example 1,
$X_1\subseteq D(A)$.

Let us demonstrate by Example A that {\it a regularizer in
the sense \eqref{ast} may be not a regularizer in the sense \eqref{1.2}}.

In Example B a theoretical construction of a regularizer in the sense
\eqref{1.2} is given for some equations \eqref{1.1} with nonlinear 
operators.

In Section 2 a novel theoretical construction of a regularizer in the 
sense \eqref{1.2} is given for a very wide class of equations \eqref{1.1} 
with linear operators $A$.

{\bf Example A}: {\it Stable numerical differentiation.}

In this Example the results from \cite{R32} - \cite{R428} are used.
This Example is borrowed from \cite{R444}.

Consider stable numerical differentiation of noisy data.
The problem is:
\begin{equation} \label{1.3}
Au:=\int^x_0 u(s)\,ds=f(x), \quad f(0)=0, \quad 0\leq x\leq 1.
\end{equation}
The data are: $f_\delta$ and a constant $M_a$, which defines a 
compact ${\mathcal K}$, where 
$\|f_\delta-f\|\leq \delta$,
the norm is $L^\infty(0,1)$ norm, and ${\mathcal K}$ consists of the 
$L^\infty$ functions which satisfy the inequality
$\|u\|_a\leq M_a$, $a\geq 0$.
The norm
$$\|u\|_a := \sup_{\substack {x,y\in[0,1] \\ x\neq y}}
  \ \frac{|u(x)-u(y)|}{|x-y|^a}
 + \sup_{0\leq x\leq 1} |u(x)|\quad \hbox{\ if\ }  0\leq a\leq 1,
$$
$$\|u\|_a:=\sup_{0\leq x\leq 1} (|u(x)| + |u'(x)|)
 +\sup_{\substack {x,y\in[0,1] \\ x\neq y}}
  \ \frac{|u'(x)-u'(y)|}{|x-y|^{a-1}},
 \quad 1<a\leq 2.
$$
If $a>1$, then we define
\begin{equation} \label{1.4}
R(\delta)f_\delta:=
\begin{cases}
\frac{f_\delta(x+h(\delta))-f_\delta(x-h(\delta))}{2h(\delta)},
   & h(\delta)\leq x\leq 1-h(\delta),\\
\frac{f_\delta(x+h(\delta))-f_\delta(x)}{h(\delta)},
   & 0\leq x< h(\delta), \\
\frac{f_\delta(x)-f_\delta(x-h(\delta))}{h(\delta)},
   & 1-h(\delta)<x\leq 1,
\end{cases} \end{equation}
where
\begin{equation}
\label{1.5} h(\delta)=c_a\delta^{\frac{1}{a}},
\end{equation}
and $c_a$ is a constant given explicitly (cf \cite{R133}). 

We prove that
\eqref{1.4} is a regularizer for \eqref{1.3} in the sense \eqref{1.2},
and ${\mathcal K}:=\{v:\|v\|_a \leq M_a,\ a>1\}$.
In this example we do not use lower semicontinuity of the norm $\phi(v)$
and do not define $\phi$. 

Let $S_{\delta,a}:=\{ v:\|Av-f_\delta\| \leq \delta, \ \|v\|_a \leq M_a 
\}$. 
To prove that \eqref{1.4}-\eqref{1.5} is a regularizer in the sense
\eqref{1.2} we use the estimate
\begin{equation} \label{1.6}
\begin{aligned}
\sup_{\substack {v\in S_{\delta,a}}}  &  \|R(\delta) f_\delta-v\|
\leq \sup_{\substack {v\in S_{\delta,a}}}
  \{ \|R(\delta)(f_\delta-Av)\| +\| R(\delta)Av-v\| \}
\leq \frac{\delta}{h(\delta)} +M_a h^{a-1}(\delta) \leq
 \\&
\leq c_a \delta^{1-\frac{1}{a}} :=\eta(\delta)\to 0
  \hbox{\ as\ } \delta \to 0. 
\end{aligned}\end{equation}
Thus we have proved that \eqref{1.4}-\eqref{1.5} is a regularizer
in the sense \eqref{1.2}.

If $a=1$, and $M_1<\infty$, then one can prove the following result:

 {\bf Claim:}   {\it There is no
regularizer for problem \eqref{1.3} in the sense \eqref{1.2}
even if the regularizer is sought in the set of all operators,
including nonlinear ones}. 

More precisely, it is proved in \cite{R190}, 
p.345, (see also \cite{R499}, pp 197-235, where the stable numerical 
differentiation problem is discussed in detail) that
$$\inf_{R(\delta)} \ \sup_{v\in S_{\delta,1}}
 \ \|R(\delta) f_\delta - v\| \geq c>0,$$
where $c>0$ is a constant independent of $\delta$
and the infimum is taken {\it over all operators
$R(\delta)$ acting from $L^{\infty}(0,1)$
into  $L^{\infty}(0,1)$, including nonlinear ones}.

On the other hand, if $a=1$ and $M_1<\infty$, then
a regularizer in the sense \eqref{ast} does exist,
but the rate of convergence in (*) may be as slow
as one wishes, if
$u(x)$ is chosen suitably (see \cite{R133}, \cite{R499}).

{\bf Example B:} {\it Construction of a regularizer in the sense 
\eqref{1.2} for some nonlinear equations.}

Assuming A1) and A2), let us 
construct a regularizer
for \eqref{1.1} in the sense \eqref{1.2}. We use the ideas from
\cite{R444} and \cite{R428}.

Define $F_\delta(v):=\|Av-f_\delta\| + \delta\phi(v)$ and
consider the minimization problem of finding the infimum $m(\delta)$ of 
the functional $ F_\delta(v)$ on a set $ S_\delta$:
\begin{equation} \label{1.7}
m(\delta):=\inf_{v\in S_\d} F_\delta(v),
 \qquad S_\delta:=\{v:\|Av-f_\delta\| \leq \delta,
 \ \phi(v)\leq c\}.
\end{equation}
Here 
$${\mathcal K}={\mathcal K}_c:=\{ v:\phi(v)\leq c\}.$$
The constant $c>0$ can be chosen arbitrary large and fixed at
the beginning of the argument, and then one can choose a smaller
constant $c_1$, specified below.
Since $F_\delta(u)=\delta + \delta\phi(u):=c_1\delta$,
$c_1:=1+\phi(u)$, where $u$ solves \eqref{1.1}, one concludes that
\begin{equation} \label{1.8}
 m(\delta)\leq c_1\delta.
\end{equation}
Let $v_j$ be a minimizing sequence and
$F_\delta(v_j)\leq 2m(\delta)$.
Then $\phi(v_j)\leq 2 c_1$. By assumption A2),
as $j\to\infty$, one has:
\begin{equation}  \label{1.9}
 v_j\to v_\delta,
 \qquad \phi(v_\delta) \leq 2c_1.
\end{equation}
Take $\delta=\delta_m\to 0$ and denote
$v_{\delta_m}:=w_m$.
Then \eqref{1.9} and Assumption A2) imply the existence
of a subsequence, denoted again $w_m$, such that:
\begin{equation} \label{1.10}
 w_m\to w,\qquad A(w_m)\to A(w), \qquad \|A(w)-g\|=0.
\end{equation}
Thus $A(w)=g$ and, since $A$ is injective by Assumption A1), 
it follows that $w=u$, where 
$u$ is the
unique solution to \eqref{1.1}.

Define now $R(\delta)f_\delta$ by the formula 
$R(\delta)f_\delta:=v_\delta$, where $v_\delta$ is defined in 
\eqref{1.9}.

\begin{theorem}
$R(\delta)$ is a regularizer for problem \eqref{1.1} in the
sense \eqref{1.2}.
\end{theorem}

\begin{proof}
Assume the contrary:
\begin{equation}\label{1.11}
  \sup_{v\in S_\delta} \|R(\delta)f_\delta - v\|
  = \sup_{v\in S_\delta} \|v_\delta-v\| \geq \gamma >0,
\end{equation}
where $\gamma>0$ is a constant independent of $\delta$.
Since $\phi(v_\delta)\leq 2c_1$ by \eqref{1.9}, and
$\phi(v)\leq c$, one can choose convergent in $X$ sequences
$w_m:=v_{\delta_m}\to \tilde w$, $\delta_m\to 0$, and $v_m\to\tilde 
v$,
such that $\|w_m-v_m\|\geq \frac{\gamma}{2}$,
$\| \tilde w -\tilde v\|\geq\frac{\gamma}{2}$,
and $A(\tilde w)=g$, $A(\tilde v)=g$.
By the injectivity of $A$ it follows that $\tilde w=\tilde 
v=u$. This contradicts
the inequality $\|\tilde w-\tilde v\|\geq\frac{\gamma}{2}>0$.
This contradiction proves the theorem.    

The conclusions $A(\tilde w)=g$ and $A(\tilde v)=g$, that we have used 
above, follow
from the inequalities $\|A(v_\delta)-f_\delta\|\leq \delta$
and $\|A(v)-f_\delta\| \leq \delta$
after passing to the limit $\delta\to 0$, using assumption A2).
\end{proof}

\section { Construction of a regularizer in the sense
\eqref{1.2} for linear equations}

If $A$ is a linear  closed densely defined in $H$ operator, 
then $T=A^*A$ is a densely defined selfadjoint operator. Let $T_a:=T+aI$,
where  $a=const>0$. The operator $T_a^{-1}A^*$ is densely defined and 
closable. Its closure is a bounded operator, defined 
on all of $H$, and $||T_a^{-1}A^*||\leq \frac 1 
{2\sqrt{a}}.$  See \cite{R500}-\cite{R522} for details and other results.
Let $E_s$ be the resolution of the identity of the selfadjoint operator 
$T$, $d\rho:=d(E_sy,y)$, and 
${\mathcal K}:=\{u: \int_0^\infty s^{-2p}d\rho\leq k_p^2\}$, where
$p\in (0,1)$ and $k_p>0$ are constants. 

Our basic result is:

{\bf Theorem 2.1.} {\it The operator $R_\d=T_{a(\d)}^{-1}A^*$   
is a regularizer for problem \eqref{1.1} in the sense \eqref{1.2}
if $\lim_{\d\to 0} \frac {\d}{a(\d)^{1/2}}=0$ and $\lim_{\d\to 
0}a(\d)=0$. Moreover, if $a(\d)=b_p\d^{\frac 2{2p+1}},$ then
\begin{equation}\label{2.1}
  \sup_{y\in {\mathcal K}, ||Ay-f_\d||\leq \d} \|R(\delta)f_\delta - y\|
  \leq C_p\d^{\frac {2p}{2p+1}},
\end{equation}
where 
$$C_p=\frac 1{2\sqrt{b_p}}+c_pk_pb_p^p, \qquad c_p=p^p(1-p)^{1-p},
\qquad b_p:=(4pc_pk_p)^{-\frac {2}{2p+1}}.$$ 
The above choice of $a(\d)$ is optimal 
in the sense that the right-hand side of \eqref{2.2} (see below) is 
minimal for this choice of  $a(\d)$.
}

{\it Proof.} 

Let
$$\epsilon:=\sup_{y\in {\mathcal K}, ||Ay-f_\d||\leq 
\d}||T_a^{-1}A^*f_\d-y||:=\sup ||T_a^{-1}A^*f_\d-y||.$$
 Then, with $Ay=f$, one has 
$$\epsilon\leq \sup ||T_a^{-1}A^*(f_\d-f)||+\sup 
||T_a^{-1}A^*Ay-y||:=J_1+J_2,
$$
where
$$J_1\leq \frac {\d}{2\sqrt{a}},$$
and
$$J_2^2\leq \sup \{a^2||T_a^{-1}y||^2\}\leq \sup \int_0^\infty \frac 
{a^2}{(s+a)^2} d(E_sy,y).$$ 

Thus,
$$J_2^2\leq \big(\max_{s\geq 0}\frac 
{as^p}{a+s}\big)^2k_p^2=c_p^2k_p^2a^{2p},$$
because  $\max_{s\geq 0}\frac{as^p}{a+s}$ is attained at 
$s=\frac{pa}{1-p}$ and is equal to $c_pa^p$, where
$$c_p:=p^p(1-p)^{1-p},\qquad k_p^2:=\sup_{y\in {\mathcal K}}\int_0^\infty 
s^{-2p}d(E_sy,y).
$$
Consequently,
$$J_2\leq c_pk_pa^{p},$$
and
\begin{equation}\label{2.2}
\epsilon\leq \frac {\d}{2\sqrt{a}}+c_pk_pa^{p}.
\end{equation}
Minimizing the right-hand side of \eqref{2.2} with respect to $a> 0$, 
one obtains inequality \eqref{2.1}. 

The minimizer of the right-hand side of \eqref{2.2} is 
$$a=a(\d)=b_p\d^{\frac {2}{2p+1}}, \qquad 
b_p:=(4pc_pk_p)^{-\frac {2}{2p+1}},$$
and the minimum of the right-hand side of \eqref{2.2} is
$C_p\d^{\frac {2p}{2p+1}}$, where 
\begin{equation}\label{2.3}
C_p:=\frac 1{2\sqrt{b_p}}+c_pk_pb_p^p.
\end{equation}
Theorem 2.1 is proved. \hfill $\Box$

\end{document}